\documentclass[12pt]{article}
\usepackage{fullpage,amsthm,amsmath,amssymb,xspace,tikz,url,caption,subcaption,enumitem}

\newcounter{thmctr}
\newtheorem{thm}[thmctr]{Theorem}
\newtheorem{lemma}[thmctr]{Lemma}
\newtheorem{prop}[thmctr]{Proposition}

\newtheorem*{definition}{Definition}
\theoremstyle{definition}
\newtheorem*{problem}{Problem}
\newtheorem*{constr}{Construction}
\newtheorem*{remarks}{Remarks}
\theoremstyle{plain}

\newcommand{\sm}{\setminus} 
\newcommand{\B}{\mathcal{B}}
\newcommand{\Z}{\mathcal{Z}}  
\newcommand{\eps}{\varepsilon}
\newcommand{\Pa}{{\mathcal P}}

\usetikzlibrary{calc,intersections}
\tikzstyle{vertex}=[circle,fill=black,inner sep=2pt]

\newcommand{\dhruvuni}{University of Illinois at Chicago \\ mubayi@uic.edu}
\newcommand{\johnuni}{University of Illinois at Chicago \\ lenz@math.uic.edu}
\newcommand{\dhruvfoot}{\footnote{Research supported in part by  NSF Grants DMS-0969092 and DMS-1300138.}}
\newcommand{\johnfoot}{\footnote{Research partly supported by NSA Grant H98230-13-1-0224.}}
\newcommand{\richarduni}{University of Birmingham \\	r.mycroft@bham.ac.uk}
\newcommand{\richardfoot}{\footnote{Research supported in part by EPSRC Grant EP/M011771/1.}}

\title{Hamilton cycles in quasirandom hypergraphs\footnote{Accepted for publication in Random Structures \& Algorithms.}}

\author{John Lenz \johnfoot \\ \johnuni \and Dhruv Mubayi \dhruvfoot \\ \dhruvuni \and Richard
Mycroft \richardfoot \\ \richarduni}

\begin{document}

\maketitle

\begin{abstract}
We show that, for a natural notion of quasirandomness in $k$-uniform hypergraphs, any quasirandom $k$-uniform hypergraph on $n$ vertices with constant edge density and minimum vertex degree $\Omega(n^{k-1})$ contains a loose Hamilton cycle. We also give a construction to show that a $k$-uniform hypergraph satisfying these conditions need not contain a Hamilton $\ell$-cycle if $k-\ell$ divides $k$. The remaining values of $\ell$ form an interesting open question.
\end{abstract}

\section{Introduction}\label{intro}

A \emph{$k$-uniform hypergraph}, or \emph{$k$-graph} $H$ consists of a set of vertices $V(H)$ and a set of edges $E(H)$, where each edge
consists of $k$ vertices. We say that a $k$-graph $C$ is an \emph{$\ell$-cycle} if its vertices can be
cyclically ordered in such a way that each edge of $C$ consists of $k$ consecutive vertices, and so
that each edge intersects the subsequent edge in $\ell$ vertices. This generalises the notion of a
cycle in a graph in a natural manner, though various other definitions of cycles in hypergraphs have also
been considered, such as a Berge cycle~\cite{hh-bermond78}.  Note in particular that each edge of an
$\ell$-cycle $k$-graph $C$ has $k-\ell$ vertices which were not contained in the previous edge, so
the number of vertices of $C$ must be divisible by $k-\ell$.  We say that a $k$-graph $H$ on $n$
vertices contains a \emph{Hamilton $\ell$-cycle} if it contains an $n$-vertex $\ell$-cycle as a
subgraph; as before, this is only possible if $k-\ell$ divides $n$. We refer to $1$-cycles and $(k-1)$-cycles as \emph{loose cycles} and \emph{tight cycles} respectively, and in the same way refer to \emph{loose Hamilton cycles} and \emph{tight Hamilton cycles}.

Finding sufficient conditions which ensure that a $k$-graph contains a Hamilton $\ell$-cycle (or other spanning structure)
has been a highly active area of research in recent years, with particular interest in minimum degree conditions and quasirandomness conditions. 

\subsection{Minimum degree conditions}
Sufficient minimum degree conditions which ensure that a $k$-graph contains a Hamilton $\ell$-cycle 
can be seen as hypergraph analogues of the well-known Dirac's theorem~\cite{dirac52}, which states that 
any graph $G$ on $n \geq 3$ vertices with $\delta(G) \geq n/2$ contains a Hamilton cycle. 
For a $k$-graph $H$ and a set $S \subseteq V(H)$, the \emph{degree} of $S$, denoted $d_H(S)$ or $d(S)$ (when $H$ is clear from the context), is the the number of edges of $H$ which 
contain $S$ as a subset. If $S = \{v\}$ is a singleton then we write simply $d_H(v)$ or $d(v)$ rather than $d(\{v\})$.
The \emph{minimum $s$-degree} of $H$, written $\delta_s(H)$, is the 
minimum of $d(S)$ taken over all $s$-sets of vertices of $H$. In particular we refer to the 
minimum $1$-degree and the minimum $k-1$ degree of $H$ as the \emph{minimum vertex degree} 
and \emph{minimum codegree} of $H$ respectively. The following theorem gives an Dirac-type result for $k$-graphs with high minimum codegree; simple constructions show that for any $k$ and $\ell$ this minimum codegree condition is best possible up to the $\eta n$ error term. 

\begin{thm} \label{codeg}
  For any $k \geq 3$, $1 \leq \ell \leq k-1$ and $\eta > 0$, there exists~$n_0$ such that if $n \geq
  n_0$ is divisible by $k-\ell$ and $H$ is a $k$-graph on $n$ vertices with $$\delta_{k-1}(H) \geq 
  \begin{cases}
    \left( \frac{1}{2} + \eta \right) n& \mbox{ if $k-\ell $ divides $k$,} \\
    \left(\frac{1}{\lceil 
    \frac{k}{k-\ell} \rceil(k-\ell)}+\eta\right) n & \mbox{otherwise,}
  \end{cases} 
  $$
  then $H$ contains a Hamilton $\ell$-cycle.
\end{thm}

The case $\ell=k-1$ of Theorem~\ref{codeg} was proved by R\"odl, Ruci\'nski and
Szemer\'edi~\cite{hh-rodl08}, confirming a conjecture of Katona and
Kierstead~\cite{hh-katona99}, by their innovative `absorbing method' (the approach we follow in this
paper). The same authors then showed that for the case $k=3$ and $\ell = 2$ the $\eta n$ error term can be removed~\cite{hh-rodl06}.
The remaining cases of Theorem~\ref{codeg} with $k-\ell$ divides $k$ follow immediately from the case $\ell = k-1$ , since if
$k-\ell$ divides~$n$ then any $(k-1)$-cycle of order $n$ contains an $\ell$-cycle on the same vertex
set as a subgraph. The remaining cases of Theorem~\ref{codeg} were subsequently proven in a sequence of papers by K\"uhn and Osthus~\cite{hh-kuhn06}, Keevash, K\"uhn, Mycroft, and Osthus~\cite{hh-keevash11}, H\`an and Schacht~\cite{hh-han10} and K\"uhn, Mycroft, and Osthus~\cite{hh-kuhn10}, while more recently Han and Zhao~\cite{han15} showed that the $\eta n$ error term can also be removed for $\ell < k/2$.


Much less is known for other minimum degree conditions. In particular, an analogous result to 
Theorem~\ref{codeg} for minimum vertex degree conditions is only known for $k=3$ and $\ell = 1$. 
This is the following theorem due to Han and Zhao~\cite{hh-han14} (an asymptotic version of this 
result was given previously by Bu\ss, H\`an and Schacht~\cite{hh-bus13}).

\begin{thm} \label{vdeg}
  There exists~$n_0$ such that if $n \geq n_0$ is even and $H$ is a $3$-graph on $n$ vertices with
  $$\delta_{1}(H) \geq \binom{n-1}{2} - \binom{\lfloor \frac{3n}{4}\rfloor}{2} + 2 - n \bmod{2}$$
  then $H$ contains a Hamilton $1$-cycle.
\end{thm}

\subsection{Quasirandomness conditions} \label{sec:introquasicond}

Another type of sufficient condition for the existence of a Hamilton $\ell$-cycle in a $k$-graph $H$ is a quasirandomness condition, where we assume that $H$ satisfies some property (or properties) typical of a random $k$-graph on the same vertex set; there are many candidate properties which could be considered. In this vein, Frieze and Krivelevich~\cite{Frieze12a} defined a form of quasirandomness for $k$-graphs termed $(p, \eps)$-regularity, and proved for $1 \leq \ell \leq k/2$ that not only must any $(p, \eps)$-regular $k$-graph $H$ contain a Hamilton $\ell$-cycle, but in fact that any such $H$ contains a collection of edge-disjoint Hamilton $\ell$-cycles covering almost all of the edges of $H$. Frieze, Krivelevich and Loh~\cite{Frieze12b} then proved an analogous result for tight Hamilton cycles in $3$-graphs, following which Bal and Frieze~\cite{Bal12} proved an analogous result for Hamilton $\ell$-cycles in $k$-graphs with $k/2 < \ell \leq k-1$ (each of the latter two papers used a somewhat different definition of $(p, \eps)$-regularity tailored to the problem in question).

However, the notions of quasirandomness used for these results are very strong; each involves a 
`generalised codegree condition' which gives the approximate size of the intersection of the 
neighbourhoods of a small number of sets of vertices. Weaker notions of quasirandomness in hypergraphs have also been studied; Lenz and Mubayi~\cite{hqsi-lenz-poset12} determined the poset of implications between many such notions, and also demonstrated that each such notion is equivalent to the existence of a large spectral gap for an appropriate definition of first and second eigenvalues~\cite{hqsi-lenz-quasi12}. In this paper we consider the weakest of these forms of quasirandomness, which Lenz and Mubayi referred to as {\tt Expand}$_p$[1+\dots+1], and which is a natural generalisation to $k$-graphs of a notion of
quasirandomness for graphs that originated in early work of
Thomason~\cite{qsi-thomason87,qsi-thomason87-2} and Chung, Graham and Wilson~\cite{qsi-chung89}. 

\begin{definition}
  Let $k \geq 2$, let $0<\mu,p<1$, and let $H$ be a $k$-graph on $n$ vertices.  We say that $H$ is \emph{$(p,\mu)$-dense}
  if for any $X_1, \dots, X_k \subseteq V(H)$ we have
  \begin{align*}
    e(X_1,\dots,X_k) \geq p |X_1| \cdots |X_k| - \mu n^k,
  \end{align*}
  where $e(X_1,\dots,X_k)$ is the number of $k$-tuples $(x_1,\dots,x_k) \in X_1 \times \cdots \times X_k$ such
  that $\{x_1,\dots,x_k\} \in H$ (note that if the sets $X_i$ overlap an edge might be counted more than
  once). We say that $H$ is an \emph{$(n,p,\mu)$ $k$-graph} if $H$ has $n$ vertices and is $(p,\mu)$-dense.
  Finally, for  $0 < \alpha < 1$, an \emph{$(n,p,\mu,\alpha)$ $k$-graph} is an $(n,p,\mu)$ $k$-graph $H$ 
  which also satisfies $\delta_1(H) \geq \alpha \binom{n}{k-1}$.
\end{definition}

Note that an $(n, p, \mu)$ $k$-graph may contain isolated vertices, so this notion of quasirandomness by itself does not imply the existence of even a single Hamilton $\ell$-cycle. Lenz and Mubayi~\cite{pp-lenz14} previously studied perfect packings in $(n, p, \mu, \alpha)$ $k$-graphs, showing that any such $k$-graph must contain a perfect $F$-packing for any fixed linear $k$-graph $F$ (for $n$ sufficiently large, $\mu$ sufficiently small, and subject to certain natural divisibility conditions); they then continued this line of research through similar packing results for various related notions of quasirandomness in~\cite{pp-lenz14-2}.

\subsection{New results} \label{sec:new}

Our main result in this paper is the following theorem, which states that any $k$-graph satisfying our quasirandomness condition, whose minimum vertex degree is not too small, must contain a loose Hamilton cycle. This is the first example of a connected spanning structure whose existence is guaranteed by this notion of quasirandomness.

\begin{thm} \label{main}
  Let $k \geq 2$. For any $0 < p, \alpha < 1$ there exist $n_0$ and $\mu >0$ such that if $H$ is an
  $(n, p, \mu, \alpha)$ $k$-graph, where $n \geq n_0$ is divisible by $k-1$, then $H$ contains a
  Hamilton $1$-cycle.
\end{thm}

\begin{remarks} \hspace{1cm}

\begin{itemize}
  \item The minimum vertex degree condition cannot be dropped from the statement of
    Theorem~\ref{main}.  Indeed, fix $p \in (0,1)$, let $f(n) = o(n)$ and consider the following $k$-graph sequence: take the
    disjoint union of the random $k$-graph $G^{(k)}(n - f(n),p)$ and a clique of size $f(n)$.  The minimum vertex
    degree is $\binom{f(n)-1}{k-1}$, there is no Hamilton $1$-cycle, and for all $\mu>0$, the $k$-graph is still
    $(p,\mu)$-dense with high probability.

  \item It is not true that a $k$-graph $H$ which satisfies the conditions of Theorem~\ref{main}
    must contain edge-disjoint Hamilton $1$-cycles covering almost all of the edges of $H$ (as was
    the case for the $(p, \eps)$-regular $k$-graphs discussed in Section~\ref{sec:introquasicond}). For example, let $H$ be formed
    from the complete $k$-graph on $n$ vertices by fixing a vertex $v$ and arbitrarily deleting
    $\tfrac{2}{3}\binom{n-1}{k-1}$ edges which contain $v$. Then, providing $n$ is sufficiently
    large compared to $\mu$, $H$ is an $(n, \tfrac{1}{2}, \mu, \tfrac{1}{4})$ $k$-graph.
    Furthermore, the size of any collection of edge-disjoint $1$-cycles in $H$ is at most the degree
    of $v$, that is, $\tfrac{1}{3}\binom{n-1}{k-1}$. Since a Hamilton $1$-cycle contains
    $\tfrac{n}{k-1}$ edges, we conclude that any collection of edge-disjoint Hamilton $1$-cycles in
    $H$ covers at most $\tfrac{n}{3(k-1)}\binom{n-1}{k-1} = \tfrac{k}{3(k-1)}\binom{n}{k}$ edges,
    that is, at most around two-thirds of the edges of $H$.

  \item The definition of $(p,\mu)$-dense cannot be changed to  $e(X) \geq p\binom{|X|}{k} -
    \mu n^k$ for all vertex sets $X$, where $e(X)$ is the number of edges in $X$.  Indeed, consider the disjoint union $G$ of two cliques of size
    $n/2$.  For all $\mu > 0$ and $X \subseteq V(G)$, we have $e(X) \geq
    2^{-k} \binom{|X|}{k} - \mu n^k$ for large $n$.  Clearly $G$ has minimum vertex degree $\Omega(n^{k-1})$ and no Hamilton
    $1$-cycle.

  \item Our proof of Theorem~\ref{main} is valid for graphs (i.e.~when $k=2$), but Theorem~\ref{main} was already known in this case. Indeed, K\"uhn, Osthus and Treglown~\cite{KOT} showed that any large `robust expander' graph with high minimum degree must contain a Hamilton cycle, and it is straightforward to show that any $(n, p, \mu, \alpha)$ $2$-graph has this form (actually their result was stated for directed graphs, but the analogue for undirected graphs follows immediately). Very recently a proof of this result which does not use graph regularity (and so applies for much smaller graphs) was given by Lo and Patel~\cite{LP}. Other `expansion' properties which ensure the existence of Hamilton cycles in graphs were established by Hefetz, Krivelevich and Szab\'o~\cite{HKS}. Our results can be seen as providing a hypergraph analogue of these results; our quasirandomness condition is an `expansion' property which guarantees the existance of a Hamilton cycle.

  \item We do not treat the case $p = o(1)$ as $n \rightarrow \infty$, but the techniques likely
    extend to the sparse setting.
\end{itemize}
\end{remarks}

It is natural to ask whether analogous results to Theorem~\ref{main} hold for Hamilton $\ell$-cycles when $\ell \geq 2$. Since consecutive edges of a Hamilton $\ell$-cycle intersect in $\ell$ vertices, a necessary prerequisite for this is that we strengthen the minimum vertex degree condition $\delta_1(H) \geq \alpha \binom{n}{k-1}$ of Theorem~\ref{main} to a minimum $\ell$-degree condition $\delta_\ell(H) \geq \alpha \binom{n}{k-\ell}$. Indeed, in Section~\ref{sec:constr} we give a construction of an $(n, p, \mu)$ $k$-graph $H$ with $\delta_{j}(H) \geq \alpha\binom{n}{k-j}$ for any $1 \leq j \leq \ell-1$ which contains a vertex $x$ such that the intersection of any edge containing $x$ and any edge not containing $x$ has size at most $\ell-1$; it follows that $H$ contains no Hamilton $\ell$-cycle. 

Interestingly, in the case where $k \geq 3$ and $k-\ell$ divides $k$, the analogous statement to Theorem~\ref{main} for Hamilton $\ell$-cycles does not hold even after this strengthening (i.e. assuming also that $H$ satisfies the minimum $\ell$-degree condition $\delta_\ell(H) \geq \alpha \binom{n}{k-\ell}$). Indeed, by adapting a construction of Lenz and Mubayi~\cite{pp-lenz14, pp-lenz14-2} we prove the following proposition.

\begin{prop} \label{prop:constr}
Let $k$ and $\ell$ be integers such that $k \geq 3$, $2 \leq \ell \leq k-1$ and $k-\ell$ 
divides $k$. Then for any $\mu > 0$ there exists $n_0$ such that for any $n \geq n_0$ which is 
divisible by $2k$, there is a $k$-graph $H$ on $n$ vertices such that
  \begin{enumerate}[noitemsep, label=(\alph*)]
    \item $H$ is $(2^{-\binom{k}{\ell}},\mu)$-dense,
    \item $\delta_j(H) \geq (2^{-\binom{k}{\ell}} - \mu) \binom{n}{k-j}$ for any $1 \leq j \leq \ell$,
    \item $H$ does not contain a Hamilton $\ell$-cycle.
  \end{enumerate}
\end{prop}

In particular, we cannot guarantee the existence of a tight Hamilton cycle in an $(n, p, \mu, \alpha)$ $k$-graph $H$ even if we also assume that $H$ has minimum codegree $\delta_{k-1}(H) \geq \alpha n$. 

In summary, we have shown that any $(n, p, \mu)$ $k$-graph $H$ with $\delta_\ell(H) \geq \alpha \binom{n}{k-1}$ must contain a Hamilton $\ell$-cycle if $\ell = 1$, whilst for $k \geq 3$ this statement is false if $k-\ell$ divides~$k$. This leaves the remaining cases as an interesting open problem.

\begin{problem}
 Fix $\ell>1$ and  $k \ge 4$ such that $k-\ell$ does not divide $k$.  Is the following statement true?
For all $0 < p, \alpha < 1$ there exist $n_0$ and $\mu >0$ such that if $H$ is an
  $(n, p, \mu)$ $k$-graph with $\delta_\ell(H) \geq \alpha \binom{n}{k-\ell}$, 
where $n \geq n_0$ is divisible by $k-\ell$, then $H$ contains a Hamilton $\ell$-cycle. 
\end{problem}

\subsection{Structure and notation of this paper}

The remainder of this paper is organized as follows. In Section~\ref{sec:absorbing}, we outline the `absorbing method' introduced by R\"odl, Ruci\'nski and Szemer\'edi, and state three key lemmas needed to apply this method in the context of Theorem~\ref{main}. We then combine these key lemmas to prove Theorem~\ref{main}. In Section~\ref{sec:keylemmas} we prove each of the three key lemmas, for which our main tool is an extension lemma for quasirandom $k$-graphs developed by Lenz and Mubayi~\cite{pp-lenz14}. Finally, in Section~\ref{sec:constr} we prove Proposition~\ref{prop:constr}.

Throughout this paper we identify a $k$-graph $H$ with its edge set, for example writing $|H|$ for the number of edges of $H$ and $e \in H$ to mean $e \in E(H)$. For a set $A$ we write $\binom{A}{k}$ to denote the collection of subsets of $A$ of size $k$. We omit floors and ceilings wherever these do not affect the argument.

\section{The absorbing method} \label{sec:absorbing}

Loosely speaking, the absorbing method proceeds as follows to find a Hamilton cycle in a $k$-graph $H$ (formal definitions will follow this outline). First, we find an `absorbing path' in $H$. This is a path $P_0$ which can `absorb' any small set $S$ of vertices of $H$ not in $P_0$, meaning that for any such $S$ there is a path $Q$ with vertex set $V(P_0) \cup S$ which has the same endvertices as $P_0$. Next, we cover almost all of the remaining vertices of $H$ with vertex-disjoint paths $P_1, \dots, P_r$ which do not intersect $P_0$. Having done this, we find `connecting' paths $Q_0, \dots, Q_r$ which are vertex-disjoint from each other and have only endvertices in common with the paths $P_0, \dots, P_r$; these endvertices are chosen so that $P_0, Q_0, P_1, Q_1, \dots, P_r, Q_r$ forms a cycle $C$. Since the paths $P_0, P_1, \dots, P_r$ covered almost all of the vertices of $H$, only a small number of vertices of $H$ are not in $C$, so we can apply the absorbing property of $P_0$ to `absorb' these vertices and so give a Hamilton cycle in $H$. The fact that we can achieve each of these steps is guaranteed by our three key lemmas.

A \emph{loose path} $P$ is a $k$-graph whose vertices can be
linearly ordered in such a way that each edge of $P$ consists of $k$ consecutive vertices, and so
that each edge intersects the subsequent edge in precisely one vertex. Since this is the only type 
of path we will discuss in this paper, we will refer to such paths simply as paths. The \emph{length} 
of a path is the number of edges it contains. If $P$ is a path, an \emph{endvertex pair} of $P$ is a 
pair of vertices $(x,y)$ such that $x$ is a degree one vertex in the first edge of $P$ and $y$ is a degree one vertex 
in the last edge of $P$. (Note there are multiple endvertex pairs if $k > 2$.)

Our first key lemma states that, under the conditions of Theorem~\ref{main}, we can find an `absorbing path'.

\newcommand{\absorbingconst}{c_{\ref{absorbingpath} }}

\begin{lemma}[Absorbing path lemma] \label{absorbingpath}
  Fix $k \geq 2$, $0 < p,\alpha < 1$, and $0 < \eps < \frac{\alpha^4p^2}{400k^2}$. There exists
  $\absorbingconst > 0$ depending only on $p$, $\alpha$, and $k$ and there exist $\mu > 0$ and
  $n_0$ depending on $p$, $\alpha$, $k$, and $\eps$ such that the following holds.  Let $H$ be
  an $(n, p, \mu, \alpha)$ $k$-graph with $n \geq n_0$. Then there exists a path $P$ in $H$ with
  endvertex pair $(u,v)$ with at most $\eps n$ vertices, with the property that for any set $X
  \subseteq V(H) \sm V(P)$ such that $k-1$ divides $|X|$ and $|X| \leq \absorbingconst \eps n$,
  there is a path $P^*$ in $H$ such that $V(P^*) = V(P) \cup X$ and $P^*$ has endvertex pair
  $(u,v)$.
\end{lemma}

The next key lemma is a connecting lemma, stating that we can find a vertex set $B$ which is
disjoint from a given small set $A$ (i.e.\ the vertices of the absorbing path) such that for any small collection of pairs of vertices, we can use $B$ to find vertex-disjoint constant-length paths
with the given endvertex pairs.

\newcommand{\connectingconst}{c_{\ref{connecting} }}

\begin{lemma}[Connecting lemma] \label{connecting}
  Fix $k \geq 2$, $0 < p,\alpha < 1$, and $0 < \eps < \frac{\alpha^2p}{20k}$.  There exists
  $0 < \connectingconst < 1$ depending only on $p$, $\alpha$, and $k$ and there exist $\mu > 0$ and
  $n_0$ depending on $p$, $\alpha$, $k$, and $\eps$ such that the following holds.  Let $H$ be an
  $(n, p, \mu, \alpha)$ $k$-graph with $n \geq n_0$ and let $A \subseteq V(H)$ with $|A| \leq
  \eps n$.  Then there exists a set $B \subseteq V(H) \sm A$ with $|B| \leq \eps^2 n$ such
  that, for any $t \leq
  \connectingconst \eps^2 n$ and any $2t$ distinct vertices $u_1, \dots, u_t, v_1, \dots, v_t$ of $H$, 
there exist vertex-disjoint paths $Q_1, \dots,
  Q_t$ of length three such that, for each $1 \leq i \leq t$, $(u_i,v_i)$ is an endvertex pair of $Q_i$ and $V(Q_i) \subseteq B
  \cup \{u_i, v_i\}$.
\end{lemma}

Finally, the path cover lemma states that we can cover almost all of the vertices of $H$ by a
constant number of vertex-disjoint paths.
 
\begin{lemma}[Path cover lemma] \label{pathcover}
  Fix $k \geq 2$, $0 < p < 1$, and $0 < \eps < \frac{p}{2k\cdot k!}$. There exist $\mu > 0$
  and $n_0$ such that the following holds.  If $H$ is an $(n, p, \mu)$ $k$-graph with $n \geq n_0$
  then there exists a collection $\Pa$ of at most $1/\eps^3$ vertex-disjoint paths in $H$ such
  that at most $\eps^2 n$ vertices of $H$ are not covered by any path $P \in \Pa$.
\end{lemma}

We can now prove Theorem~\ref{main} by combining the three key lemmas as outlined earlier.


\begin{proof}[Proof of Theorem~\ref{main}] 
Let $\absorbingconst$ and $\connectingconst$ be the constants from Lemmas~\ref{absorbingpath}
and~\ref{connecting} respectively.  Define
\begin{align*}
  \eps := \frac{1}{2} \min \left\{ \absorbingconst, \frac{\alpha^4 p^2}{400k^2}, \frac{p}{2k \cdot
  k!}
  \right\}.
\end{align*}
(Note this is a valid definition because $\absorbingconst$ only depends on $\alpha$, $p$, and $k$.)
Also, assume $\mu > 0$ is small enough and $n_0 \geq \frac{2}{c_6\eps^5}$ is large enough to apply
Lemmas~\ref{absorbingpath},~\ref{connecting}, and~\ref{pathcover} (the latter with $(1-2\eps)n_0$ here in place of $n_0$ there and $\frac{\mu}{(1-2\eps)^k}$ here in place of $\mu$ there).

First, apply Lemma~\ref{absorbingpath} to obtain a path $P$ with endvertex pair $(u,v)$, at most
$\eps n$ vertices, and such that for any set $X \subseteq V(H) \sm V(P)$ where $k-1$ divides
$|X|$ and $|X| \leq \absorbingconst \eps n$, there is a path $P^*$ in $H$ such that $V(P^*) =
V(P) \cup X$ and $P^*$ has endvertex pair $(u,v)$.  Next, let $A = V(P)$ and apply
Lemma~\ref{connecting} to obtain a set $B \subseteq V(H) \sm A$ with $|B| \leq \eps^2 n$ such that any collection of
at most $\connectingconst \eps^2 n$ pairs of vertices can be connected using $B$.

Let $H' = H \sm (A \cup B)$ and let $n' = |V(H')|$.  Note that $n' \geq (1-2\eps) n$
and that $H'$ is an $(n',p,\frac{\mu}{(1-2\eps)^k})$ $k$-graph.  Indeed, for any $X_1, \dots,
X_k \subseteq V(H')$, since $H$ is $(n,p,\mu)$-dense we have that
\begin{align*}
  e(X_1,\dots,X_k) \geq p |X_1| \cdots |X_k| - \mu n^k \geq p |X_1| \cdots |X_k| -
  \frac{\mu}{(1-2\eps)^k} (n')^k.
\end{align*}
Apply Lemma~\ref{pathcover} to $H'$ to produce a collection of vertex-disjoint paths $\mathcal{P} =
\{P_1, \dots, P_t\}$ such that $t \leq \frac{1}{\eps^3}$ and at most $\eps^2 n' \leq
\eps^2 n$ vertices of $H'$ are not covered by paths in $\mathcal{P}$.  Let $(u_i, v_i)$ be an endvertex pair of $P_i$ for
each $i$.  Also, recall that the absorbing path $P$ has endvertex pair $(u,v)$.  By choice of $B$ we can choose vertex-disjoint paths $Q_0, \dots, Q_{t}$ such that path $Q_0$ has endvertex pair $(v, u_1)$, path $Q_{t}$ has endvertex pair $(v_t, u)$, and for each $1 \leq i \leq t-1$ the path $Q_i$ has endvertex pair $(v_i, u_{i+1})$, and such that these endvertices are the only vertices of the paths $Q_i$ which do not lie in $B$. Indeed, we can do this since the number of pairs is $t + 1 \leq
\frac{1}{\eps^3} + 1 \leq \frac{2}{\eps^3}$, which is at most $\connectingconst \eps^2 n$ by choice of $n_0$. Observe that $C = P, Q_0, P_1, Q_1, \dots, P_t, Q_t$ is then a cycle in $H$.

Let $X$ be the set of all vertices of $H$ not covered by $C$. So $X$ consists of the at most $\eps^2 n$ vertices of $H'$ not covered by $\mathcal{P}$, as well as the at most $\eps^2 n$ vertices of $B$ not covered by the paths $Q_i$. Since $\eps \leq
\absorbingconst/2$, we have $|X| \leq 2 \eps^2 n \leq \absorbingconst \eps n$. Furthermore, since $n$ and $|V(C)|$ are both divisible by $k-1$, $|X|$ is divisible by $k-1$ also. So by choice of $P$ there is a path $P^*$ in $H$ such that $V(P^*) = V(P) \cup X$ and $P^*$ has endvertex pair $(u, v)$ (i.e. the same endvertex pair as $P$). Replacing $P$ by $P^*$ in $C$ gives a loose Hamilton cycle in $H$.
\end{proof}

\section{Proofs of the key lemmas} \label{sec:keylemmas}

This section contains proofs of the three key lemmas: the connecting lemma, the absorbing path
lemma, and the path cover lemma. For both the connecting lemma and absorbing path lemma, the key
element is an extension lemma for quasirandom hypergraphs proved by Lenz and Mubayi~\cite{pp-lenz14}. We actually only need the following special case of the lemma, which is obtained from the full version~\cite[Lemma 11]{pp-lenz14} by taking $Z_{m+1} = \dots = Z_f = V(H)$; it is easily checked that equations~\textit{(1)} and~\textit{(2)} of the full version of the lemma are then both satisfied.

\begin{lemma}[Extension Lemma] \label{extension}
  Fix $k \geq 2$, $0 < p,\alpha, \gamma < 1$ and integers $0 \leq m \leq f$. Suppose that $F$ is an $f$-vertex $k$-graph with vertex set $V(F) =
  \{s_1,\dots,s_m, \linebreak[1] t_{m+1},\dots,t_f\}$ such that 
\begin{enumerate}[label=(\alph*), noitemsep]
\item any edge $e \in F$ satisfies $|e \cap \{s_1, \dots, s_m\}| \leq 1$,
\item any two distinct edges $e, e' \in F$ have $|e \cap e'| \leq 1$, and
\item any two distinct edges $e, e' \in F$ with $e \cap \{s_1, \dots, s_m\} \neq \emptyset$ and $e' \cap \{s_1, \dots, s_m\} \neq \emptyset$ satisfy $e \cap e' \cap \{t_{m+1}, \dots, t_f\} = \emptyset$.
\end{enumerate}
  Then there exist $n_0$ and $\mu > 0$ such that the following holds.  Let
  $H$ be an $(n,p,\mu,\alpha)$ $k$-graph with $n \geq n_0$ and let $y_1,\dots,y_m \in V(H)$.
  Then the number of edge-preserving injections from $V(F)$ to $V(H)$ which map $s_i$ to $y_i$ for
  each $1 \leq i \leq m$ is at least
  \begin{align*}
    \alpha^{d_F(s_1)} \cdots \alpha^{d_F(s_m)} p^{|F|-\sum_{i=1}^{m} d_F(s_i)} n^{f-m} -
    \gamma n^{f-m}.
  \end{align*}
\end{lemma}

We will also require the following well-known concentration bound on sums of indicator random
variables; the form we use is Corollary~2.3 of~\cite{janson2000random}.

\begin{lemma}[Chernoff bound] \label{lem:chernoff} 
  Let $0 < p < 1$, let $X_1,\dots,X_m$ be mutually independent indicator random variables with
  $\mathbb{P}[X_i = 1] = p$ for any $1 \leq i \leq m$, and let $X = \sum_{i = 1}^m X_i$.
  Then for any $0 < a \leq 3/2$ we have
  \begin{align*}
    \mathbb{P}[ \left| X - \mathbb{E}[X] \right| > a\mathbb{E}[X]] \leq 2 e^{-a^2\mathbb{E}[X]/3}.
  \end{align*}
\end{lemma}

\subsection{The connecting lemma} \label{sub:connecting} 

Let $u$ and $v$ be distinct vertices of a $k$-graph $H$, and let $C$ be a set of $3k-4$ vertices of~$H$. We say that $C$ is a \emph{connecting set} for the pair $(u,v)$ if
$H[C \cup \{u,v\}]$ contains a path of length three with endvertex pair $(u,v)$.

\begin{lemma} \label{lem:manyconnect}
  For any $k \geq 2$ and $0 < p,\alpha < 1$, there exist $\mu > 0$ and $n_0$ such that the following holds.
  If $H$ is an $(n, p, \mu, \alpha)$ $k$-graph with $n \geq n_0$, then for any $u, v \in V(H)$ with $u \neq v$ there are at least $\frac{1}{2} \alpha^2 p \binom{n}{3k-4}$ connecting sets for $(u,v)$.
\end{lemma}

\begin{proof}
Let $\gamma = \frac{1}{2} \alpha^2 p$, let $F$ be a path of length three, 
let $m=2$, let $f = 3k-2$ and let $(s_1, s_2)$ be an endvertex pair 
of $F$. Choose $\mu > 0$ and $n_0$ such that we can apply 
Lemma~\ref{extension} with these inputs. Then for any 
$u, v \in V(H)$ with $u \neq v$, Lemma~\ref{extension} states that 
there exist at least $\frac{1}{2} \alpha^2 p n^{3k-4}$
edge-preserving injections from $V(F)$ to $V(H)$ such that 
$s_1$ maps to $u$ and $s_2$ maps to $v$. The image of $V(F) \sm \{s_1, s_2\}$ 
under any such injection is a connecting set for $(u,v)$. Since each connecting set is given by at most $(3k-4)!$ injections, we conclude that the number of connecting sets for $(u,v)$ is at least $\frac{1}{2(3k-4)!}\alpha^2p n^{3k-4} \geq \frac{1}{2} \alpha^2 p \binom{n}{3k-4}$.
\end{proof}

\begin{proof}[Proof of Lemma~\ref{connecting}]
Let $0 < p, \alpha < 1$, let $\nu = \frac{\alpha^2 p}{4}$ and let 
$\connectingconst = \frac{\nu}{50 k^2}$. Also fix $0 < \eps < \frac{\nu}{5k} = \frac{\alpha^2 p}{20k}$, and 
choose $\mu >0$ small enough and $n_0 \geq 10k$ large enough to apply 
Lemma~\ref{lem:manyconnect} and for the union bound later in the proof.  
Let $H$ be an $(n, p, \mu, \alpha)$ $k$-graph with $n \geq n_0$ and fix 
$A \subseteq V(H)$ with $|A| \leq \eps n$. Define $a = |A|$ and 
$q = \frac{\eps^2 n}{20k} \binom{n}{3k-4}^{-1}$.

Now form a collection $\B \subseteq \binom{V(H) \sm A}{3k-4}$ by including each
element of $\binom{V(H) \sm A}{3k-4}$ with probability $q$ and independently of all other choices.  The expected
size of $\B$ is $q \binom{n - a}{3k-4} \leq \frac{\eps^2  n}{20k}$ so, by Markov's
inequality, with probability at least $3/4$ we have $|\B| \leq \frac{ \eps^2 n}{5k}$.
Similarly, the expected number of ordered pairs of elements from $\B$ which intersect is at most 
$$q^2 \binom{n-a}{3k-4} (3k-4) \binom{n-a}{3k-5} = q^2 (3k-4)\frac{3k-4}{n-a-3k+5} \binom{n-a}{3k-4}^2 \leq \frac{\eps^4 n}{20}$$ 
so, by 
Markov's inequality, with probability at least $3/4$ at most 
$\frac{\eps^4 n}{5} < \frac{\nu^2\eps^2n}{125k^2}$ members of $\B$ intersect another 
member of $\B$.

For each pair of distinct vertices $u, v \in V(H)$ choose a collection $\Gamma_{u,v}$ of $\nu \binom{n}{3k-4}$ connecting sets $C$ for $(u, v)$ with $C \cap A = \emptyset$. This is possible since by Lemma~\ref{lem:manyconnect} there are at least $2\nu \binom{n}{3k-4}$ connecting sets for
$(u,v)$, and at most $|A| \binom{n}{3k-5} \leq \eps n \frac{3k-4}{n-3k+5} \binom{n}{3k-4} < \nu \binom{n}{3k-4}$ of these sets contain a vertex of $A$. 
Then for any $u$ and $v$ the expected size of $\Gamma_{u,v} \cap \B$ is $q \nu \binom{n}{3k-4} = \frac{\nu \eps^2 n}{20k}$, so by Lemma~\ref{lem:chernoff} we have
\begin{align*}
  \mathbb{P}\left[ \Big| |\Gamma_{u,v} \cap \B| - \frac{\nu \eps^2 n}{20k} \Big| > \frac{\nu
  \eps^2 n}{40k} \right] \leq 2e^{-\nu \eps^2 n/480k}.
\end{align*}
Taking a union bound over all of the $\binom{n}{2}$ pairs $(u,v)$, we conclude that with probability at least $3/4$ the collection $\B$ satisfies $|\Gamma_{u,v} \cap \B| \geq \frac{\nu \eps^2 n}{40k}$ for every $u \neq v$. 

We may therefore fix a collection $\B$ which satisfies each of the 
three described events of probability at least  $3/4$. 
We then form $\B' \subseteq \B$ by deleting from $\B$ the at most 
$\frac{\nu^2 \eps^2 n}{125k^2}$ members of $\B$ which intersect another 
member of $\B$. The family $\B'$ then satisfies $|\B'| \leq |\B| \leq 
\frac{\eps^2 n}{5k}$ and has the property that for any distinct 
vertices $u, v \in V(H)$ at least $\frac{\nu \eps^2 n}{40k} - \frac{\nu^2\eps^2n}{125 k^2} \geq \frac{\nu\eps^2n}{50k} \geq 
\connectingconst \eps^2 n$ members of $\B'$ are connecting sets for 
$(u, v)$. Let $B = \bigcup \B'$, so $|B| = (3k-4)|\B'| \leq \eps^2 n$. Finally, for any $t \leq \connectingconst \eps^2 n$ and any $2t$ distinct vertices $u_1, \dots, u_t, v_1, \dots, v_t$ of $H$, we can 
greedily choose for each pair $(u_i, v_i)$ a unique member $C_i$ of $\B'$ which is a connecting set for $(u_i, v_i)$; the chosen connecting sets give the required paths.
\end{proof}

\subsection{The absorbing path lemma} \label{sub:absorbing} 

To prove the absorbing path lemma, we will apply Lemma~\ref{extension} with the following
$k$-graph $F$.

\newcommand{\sizeF}{3k^2 - 7k + 5}
\newcommand{\sizeFprime}{3k^2 - 8k + 6} 
\newcommand{\sizeFprimeminusone}{3k^2 - 8k + 5} 
\newcommand{\edgesF}{5k-7}
\newcommand{\edgesFwiths}{2k-2}  
\newcommand{\edgesFwithouts}{3k-5} 

\begin{lemma} \label{lem:constrofF}
For any $k \geq 2$ there exists a $k$-graph $F$ with $\sizeF$ vertices and $\edgesF$ edges, a set $S \subseteq
  V(F)$ of $k-1$ vertices, and distinct vertices $u, v \in V(F)$, such that
  \begin{enumerate}[label=(\alph*), noitemsep]
    \item $F$ contains a path with endvertex pair $(u,v)$ as a spanning subgraph,
    \item $F \sm S$ contains a path with endvertex pair $(u, v)$ as a spanning subgraph,
    \item no edge $E \in F$ has $|E \cap S| \geq 2$,
    \item for any distinct edges $E_1, E_2 \in F$ we have $|E_1 \cap E_2| \leq 1$, and
    \item for any distinct edges $E_1, E_2 \in F$ with $E_1 \cap S \neq \emptyset$ and $E_2 \cap S \neq \emptyset$ we have $(E_1 \cap E_2) \sm S = \emptyset$.
  \end{enumerate}
\end{lemma}

\begin{proof} 
To define the $k$-graph $F$, we first define a $2$-graph $F'$ and then form $F$ by adding
$k-2$ vertices to each edge of $F'$.  Let $V(F') = \{a_i, s_i, t_i : 1 \leq i \leq k-1\} \cup
\{b_i : 1 \leq i \leq k-2\}$ and $E(F') = \{ a_is_i, s_ib_i, a_it_i, t_ib_i, b_ia_{i+1} : 1 \leq i
\leq k-2\} \cup \{ a_{k-1}s_{k-1}, a_{k-1}t_{k-1}, s_{k-1}t_{k-1} \}$.  Also, let $u = a_1$, $v =
t_{k-1}$, and $S = \{s_1,\dots,s_{k-1}\}$.  Note that $F'\sm S$ has a spanning $(u,v)$-path and $F'$
has a $(u,v)$-path which covers all vertices except $t_1,\dots,t_{k-2}$.

\begin{figure}[ht] 
  \center
  \begin{tikzpicture}
    \node (a1) at (0,0) [vertex, label=left:{$u=a_1$}] {};
    \node (s1) at (1,2) [vertex, label=above:$s_1$] {};
    \node (t1) at (1,-2) [vertex, label=below:$t_1$] {};
    \node (b1) at (2,0) [vertex, label=below:$b_1$] {};
    \node (a2) at (3.5,0) [vertex, label=below:$a_2$] {};
    \node (s2) at (4.5,2) [vertex, label=above:$s_2$] {};
    \node (t2) at (4.5,-2) [vertex, label=below:$t_2$] {};
    \node (b2) at (5.5,0) [vertex, label=below:$b_2$] {};
    \node (a3) at (6.5,0) [vertex, label=below:$a_3$] {};

    \node at (7.25,0) {$\dots$};

    \node (bk3) at (8,0) [vertex, label=-80:$b_{k-3}$] {};
    \node (ak2) at (9,0) [vertex, label=right:$a_{k-2}$] {};
    \node (sk2) at (10,2) [vertex, label=above:$s_{k-2}$] {};
    \node (tk2) at (10,-2) [vertex, label=below:$t_{k-2}$] {};
    \node (bk2) at (11,0) [vertex, label=-80:$b_{k-2}$] {};
    \node (ak1) at (13,0) [vertex, label=95:$a_{k-1}$] {};
    \node (sk1) at (14,2) [vertex, label=above:$s_{k-1}$] {};
    \node (tk1) at (14,-2) [vertex, label=below:{$v = t_{k-1}$}] {};

    \draw (a1) -- (s1) -- (b1) -- (a2);
    \draw (a2) -- (s2) -- (b2) -- (a3) -- +(0.25,0.5);
    \draw (a1) -- (t1) -- (b1);
    \draw (a2) -- (t2) -- (b2);
    \draw (a3) -- +(0.25,-0.5);

    \draw (bk3) -- +(-0.25,0.5);
    \draw (bk3) -- +(-0.25,-0.5);
    \draw (bk3) -- (ak2) -- (sk2) -- (bk2) -- (ak1);
    \draw (ak1) -- (sk1) -- (tk1);
    \draw (ak2) -- (tk2) -- (bk2);
    \draw (ak1) -- (tk1);
  \end{tikzpicture}
  \caption{The graph $F'$.}
\end{figure}
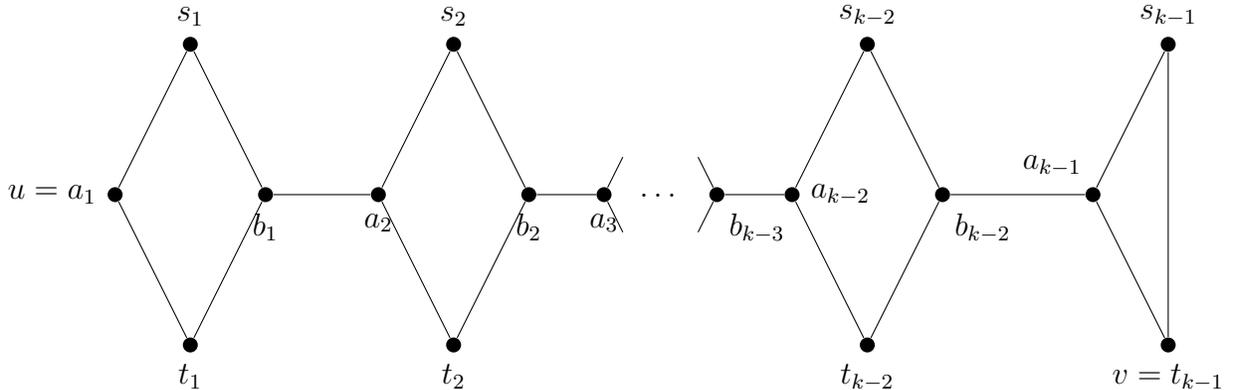 

Let $V(F) = V(F') \cup \{c_{i,j} : 0 \leq i \leq 2k-4, 1 \leq j \leq k-2\} \cup \{d_{i,j} : 1 \leq
i,j \leq k-2\}$.  We now insert exactly $k-2$ vertices into each edge of $F'$ to form the hyperedges
of $F$.  For example, into the edge $a_1s_1 \in E(F')$, we insert the vertex set $X_0 = \{c_{0,1},
c_{0,2}, \dots, c_{0,k-2}\}$.  In general, define
\begin{itemize}
  \setlength{\itemsep}{1pt}
  \setlength{\parskip}{0pt}
  \setlength{\parsep}{0pt}
  \item $C := \{c_{i,j} : 0 \leq i \leq 2k-4, 1 \leq j \leq k-2\}$,
  \item $D := \{d_{i,j} : 1 \leq i,j \leq k-2\}$,
  \item $W_i := \{d_{i,1}, \dots, d_{i,k-2}\}$ for all $1 \leq i \leq k-2$,
  \item $X_i := \{c_{i,1}, \dots, c_{i,k-2} \}$ for all $0 \leq i \leq 2k-4$,
  \item $Y_i := \{c_{i+1,1}, c_{i+2,2}, \dots, c_{i+j,j}, \dots, c_{i+k-2,k-2}\}$ for all $0 \leq i
    \leq 2k-4$, where the first index of $c$ is taken modulo $2k-3$, and
  \item $Z := \{t_1,\dots,t_{k-2}\}$.
\end{itemize}
Let the edges of $F$ be the following:
\begin{itemize}
  \setlength{\itemsep}{1pt}
  \setlength{\parskip}{0pt}
  \setlength{\parsep}{0pt}
  \item $\{a_i, s_i\} \cup X_{2i-2}$ for all $1 \leq i \leq k-1$,
  \item $\{a_i, t_i\} \cup Y_{2i-2}$ for all $1 \leq i \leq k-1$,
  \item $\{s_i, b_i\} \cup X_{2i-1}$ for all $1 \leq i \leq k-2$,
  \item $\{t_i, b_i\} \cup Y_{2i-1}$ for all $1 \leq i \leq k-2$,
  \item $\{b_i, a_{i+1}\} \cup W_i$ for all $1 \leq i \leq k-2$,
  \item $\{s_{k-1}, t_{k-1}\} \cup Z$.
\end{itemize}

\newcommand{\bubbleedge}[4]{
    \path[name path=line] (#2) -- (#3) node[midway] {#1};
    \draw[rotate=#4, name path=ellip] ($ (#2)!0.5!(#3) $) ellipse (25pt and 15pt);
    \draw[name intersections={of=line and ellip}] (#2) -- (intersection-2) (intersection-1) -- (#3);
}

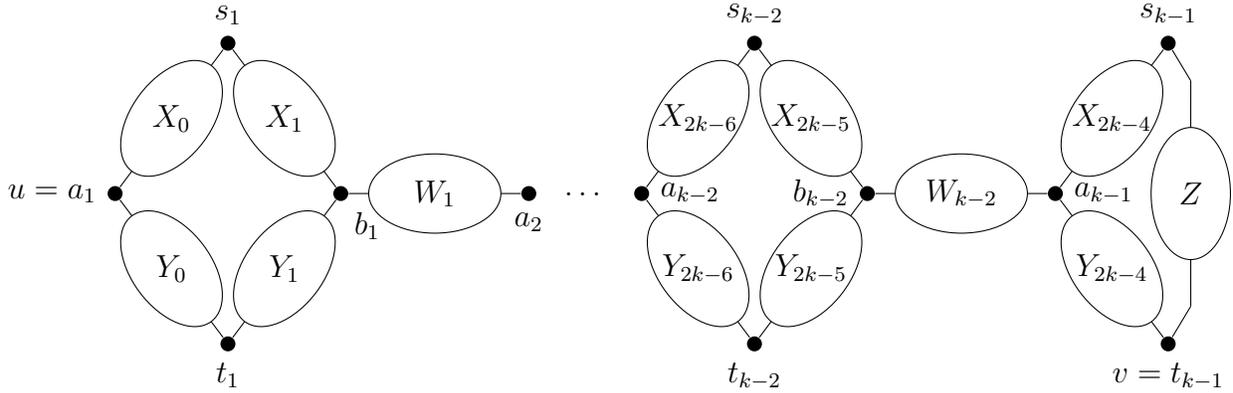
\begin{figure}[ht] 
  \center
  \begin{tikzpicture}
    \node (a1) at (0,0) [vertex, label=left:{$u=a_1$}] {};
    \node (s1) at (1.5,2) [vertex, label=above:$s_1$] {};
    \node (t1) at (1.5,-2) [vertex, label=below:$t_1$] {};
    \node (b1) at (3,0) [vertex, label=-70:$b_1$] {};
    \node (a2) at (5.5,0) [vertex, label=below:$a_2$] {};

    \node at (6.25,0) {$\dots$};

    \node (ak2) at (7,0) [vertex, label=right:$a_{k-2}$] {};
    \node (sk2) at (8.5,2) [vertex, label=above:$s_{k-2}$] {};
    \node (tk2) at (8.5,-2) [vertex, label=below:$t_{k-2}$] {};
    \node (bk2) at (10,0) [vertex, label=left:$b_{k-2}$] {};
    \node (ak1) at (12.5,0) [vertex, label=right:$a_{k-1}$] {};
    \node (sk1) at (14,2) [vertex, label=above:$s_{k-1}$] {};
    \node (tk1) at (14,-2) [vertex, label=below:{$v = t_{k-1}$}] {};

    \bubbleedge{$X_0$}{a1}{s1}{53} 
    \bubbleedge{$Y_0$}{t1}{a1}{-53}
    \bubbleedge{$X_1$}{b1}{s1}{-53}
    \bubbleedge{$Y_1$}{t1}{b1}{53}
    \bubbleedge{$W_1$}{b1}{a2}{0}

    \bubbleedge{$X_{2k-6}$}{ak2}{sk2}{53}
    \bubbleedge{$Y_{2k-6}$}{tk2}{ak2}{-53}
    \bubbleedge{$X_{2k-5}$}{bk2}{sk2}{-53}
    \bubbleedge{$Y_{2k-5}$}{tk2}{bk2}{53}
    \bubbleedge{$W_{k-2}$}{bk2}{ak1}{0}

    \bubbleedge{$X_{2k-4}$}{ak1}{sk1}{53}
    \bubbleedge{$Y_{2k-4}$}{tk1}{ak1}{-53}

    \coordinate (sk1prime) at (14.3, 1.5);
    \coordinate (tk1prime) at (14.3, -1.5);
    \draw (sk1) -- (sk1prime) (tk1prime) -- (tk1);
    \bubbleedge{$Z$}{tk1prime}{sk1prime}{90}
  \end{tikzpicture}
  \caption{The $k$-graph $F$ (note that the sets of added vertices are not disjoint).}
  \label{fig:hypergraphF}
\end{figure} 

We now verify that the properties stated in the lemma hold for $F$.  Recall that $V(F) = V(F') \cup
C \cup D$, the sets $X_0,\dots,X_{2k-4}$ partition $C$, the sets $Y_0, \dots, Y_{2k-4}$ also
partition $C$, and $W_1,\dots, W_{k-2}$ partition $D$. First note that, by construction, no two vertices of $S$ lie in the same hyperedge of $F$, so we have (c).

For (a), recall that $F'$ contains a $(u,v)$-path covering all vertices except the vertices in $Z$. The vertices in $Z$ are inserted into the edge $s_{k-1}t_{k-1}$, the sets $X_0,\dots,X_{2k-4}$ partition $C$,
    and the sets $W_1,\dots,W_{k-2}$ partition $D$ so the corresponding hyperedges in $F$ form a spanning
    path with endvertex pair $(u,v)$.  In Figure~\ref{fig:hypergraphF}, the path consists of
    the upper hyperedges.

For (b), recall that $F' \sm S$ contains a spanning $(u,v)$-path and similarly the corresponding hyperedges in $F$ form
    a spanning path with endvertex pair $(u,v)$. Indeed, $Y_0,\dots,Y_{2k-4}$ partition $C$ and
    $W_1,\dots,W_{k-2}$ partition $D$, so all vertices of $V(F') \sm S$ are used exactly once in this
    path.  In Figure~\ref{fig:hypergraphF}, the path consists of the lower hyperedges. 

For (d), consider two distinct hyperedges $E_1$ and $E_2$ of $F$ and let $E'_1$ and $E'_2$ be the
    corresponding edges in $F'$ (that is, $E_1$ was formed by adding vertices to $E'_1$ and similarly
    for $E'_2$ and $E_2$).  If $E_1$ or $E_2$ was formed by inserting a $W_i$, then $|E_1
    \cap E_2| = |E'_1 \cap E'_2| \leq 1$.  Indeed, each vertex in $W_i$ is inserted into at
    most one hyperedge so will never contribute to the intersection, and $F'$ is a graph so $|E'_1
    \cap E'_2| \leq 1$. Now suppose that $E_1$ or $E_2$ was formed by inserting $Z$, say $E_1 = \{s_{k-1}, t_{k-1} \} \cup Z$. If $Z \cap E_2 = \emptyset$ then similarly we have $|E_1 \cap E_2| = |E_1' \cap E'_2| \leq 1$. On the other hand, if $Z \cap E_2 \neq \emptyset$ then we have $t_i \in E_2$ for some $1 \leq i \leq k-2$, so either $E_2 = \{a_i, t_i\} \cup Y_{2i-2}$ or $E_2 = \{t_i, b_i\} \cup Y_{2i-1}$, and in either case we have $|E_1 \cap E_2| = 1$.

    Now consider when $E_1$ and $E_2$ are both formed by inserting one of the $X$s or $Y$s.  The sets
    $X_0,\dots,X_{2k-4}$ form a partition of $C$, so if $E_1$ and $E_2$ were both formed by inserting
    one of the $X$s, then $|E_1 \cap E_2| = |E'_1 \cap E'_2| \leq 1$.  Similarly, the sets
    $Y_0,\dots,Y_{2k-4}$ form a partition of $C$, so if both $E_1$ and $E_2$ were formed by inserting
    one of the $Y$s, then $|E_1 \cap E_2| = |E'_1 \cap E'_2| \leq 1$.  Thus without loss of generality assume
    $E_1$ was formed by inserting $X_i$ and $E_2$ was formed by inserting $Y_{\ell}$.  By construction
    we know that $|X_{i} \cap Y_{\ell}| \leq 1$ since $X_i \cap Y_{\ell} = \{ c_{i, i-\ell}\}$ if $1
    \leq i - \ell \leq k-2$ and is empty otherwise.  Thus if $E'_1 \cap E'_2 = \emptyset$, then $|E_1
    \cap E_2| \leq 1$.  If $|E'_1 \cap E'_2| = 1$, then by construction we must have $i = \ell$ since
    that is the only situation in which graph edges which insert one of the $X$s and one of the $Y$s
    share a vertex.  Since $X_i \cap Y_i = \emptyset$, we have that $|E_1 \cap E_2| \leq 1$.

Finally, for (e) consider $E_1, E_2 \in E(F)$ with $E_1 \neq E_2$, $E_1 \cap S \neq \emptyset$ and $E_2 \cap S \neq
    \emptyset$. If both $E_1$ and $E_2$ were formed by inserting $X$s, then since
    $X_0,\dots,X_{2k-4}$ is a partition of $C$ we have that $E_1$ and $E_2$ do not intersect outside
    $S$.  Now assume without loss of generality that $E_1 = \{s_{k-1}, t_{k-1}\} \cup Z$.  In this
    case, since $Z \cap C = \emptyset$, we also have that $(E_1 \cap E_2) \sm S = \emptyset$.
\end{proof} 

For the remainder of this subsection, fix such a $k$-graph $F$, a set $S \subseteq V(F)$ and distinct vertices $u, v \in V(F)$. For
any $k$-graph $H$ and any set $Y \subseteq V(H)$ with $|Y| = k-1$, we say that a set $Z \subseteq
V(H)$ with $|Z| = \sizeFprime$ is an \emph{absorbing set for $Y$} if $H[Y \cup Z]$ contains a copy
of $F$ in which $Y$ corresponds to $S$. 

\begin{lemma} \label{lem:absorbingsets}
  For any $k \geq 2$ and $0 < p,\alpha < 1$, there exist $\mu > 0$ and $n_0$ such that the following holds.  If
  $H$ is an $(n, p, \mu, \alpha)$ $k$-graph with $n \geq n_0$, then for any set $Y \subseteq V(H)$
  with $|Y| = k-1$, there are at least $\frac{1}{2} \alpha^{\edgesFwiths} p^{\edgesFwithouts}
  \binom{n}{\sizeFprime}$ absorbing sets for $Y$.
\end{lemma}

\begin{proof}
Let $\gamma = \frac{1}{2} \alpha^{\edgesFwiths} p^{\edgesFwithouts}$, 
and choose $\mu > 0$ and $n_0$ for which we can apply 
Lemma~\ref{extension} with these inputs and our chosen $k$-graph $F$. 
Then for any set $Y \subseteq V(H)$ of size $|Y| = k-1$, 
Lemma~\ref{extension} states that there are at least 
$\frac{1}{2} \alpha^{\edgesFwiths} p^{\edgesFwithouts} n^{\sizeFprime}$ 
edge-preserving injections from $V(F)$ to $V(H)$ such that the vertices 
of $S$ are mapped to the vertices of $Y$. The image of $V(F) \sm S$ under such an injection is an absorbing set for $Y$, and each absorbing set is given by at most $(\sizeFprime)!$ injections, so we conclude that the number of absorbing sets for $Y$ is at least $\frac{1}{2(\sizeFprime)!} \alpha^{\edgesFwiths} p^{\edgesFwithouts} n^{\sizeFprime} \geq \frac{1}{2} \alpha^{\edgesFwiths} p^{\edgesFwithouts}
  \binom{n}{\sizeFprime}$.
\end{proof}

\begin{proof}[Proof of Lemma~\ref{absorbingpath}]
Let $0 < p, \alpha < 1$, let $\nu = \frac{1}{2} \alpha^{\edgesFwiths} 
p^{\edgesFwithouts}$, and let $\connectingconst$ be the constant from
Lemma~\ref{connecting} for these values of $p$ and $\alpha$. Define 
$\absorbingconst = \frac{\connectingconst \nu^2}{10k^2}$, and fix 
$\eps$ with $0 < \eps < \frac{\alpha^4p^2}{400k^2}$. We will apply 
Lemma~\ref{connecting} with $\sqrt{\eps}$ in place of $\eps$; note for this 
that $0 < \sqrt{\eps} < \frac{\alpha^2 p}{20k}$. Assume that $\mu$ is
small enough and $n_0 \geq 10k^2$ is large enough for this application of 
Lemma~\ref{connecting}, and also to apply Lemma~\ref{lem:absorbingsets} 
and for the union bound later in the proof. Let $H$ be an 
$(n, p, \mu, \alpha)$ $k$-graph with $n \geq n_0$ and let
\begin{align*}
    q := \frac{\connectingconst \eps \nu n}{4k^2} \binom{n}{\sizeFprime}^{-1}.
\end{align*}
Now form a collection $\Z$ of subsets of $V(H)$ by including each set $Z \subseteq V(H)$ of
size $\sizeFprime$ at random with probability $q$ and independently of all other choices.  The
expected size of $\Z$ is $q \binom{n}{\sizeFprime} = \frac{\connectingconst \eps \nu
n}{4k^2}$ so by Markov's inequality, with probability at least $3/4$ we have $|\Z| \leq
\frac{\connectingconst \eps \nu n}{k^2}$.  Similarly, the expected number of ordered pairs of members of $\Z$ which intersect is 
$$q^2 \binom{n}{\sizeFprime} (\sizeFprime) \binom{n}{\sizeFprimeminusone} 
\leq \left(\frac{c_6\eps\nu n}{4k^2}\right)^2 \frac{(3k^2-8k+6)^2}{n-(3k^2-8k+5)}
\leq \connectingconst^2 \eps^2 \nu^2 n,$$
 so by Markov's inequality, with probability at least
$3/4$ we have that at most $4\connectingconst^2 \eps^2 \nu^2 n \leq \frac{c_6\eps\nu^2n}{100k^2}$ members of $\Z$ intersect another member of $\Z$.

For each $Y \subseteq V(H)$ of size $|Y| = k-1$, choose a collection
$\Gamma_Y$ of $\nu \binom{n}{\sizeFprime}$ absorbing sets for $Y$. 
Such a subset $\Gamma_Y$ exists by Lemma~\ref{lem:absorbingsets}. Then for any fixed $Y$ the expected size of $\Gamma_Y \cap \Z$ is $q \nu \binom{n}{\sizeFprime} = \frac{\connectingconst \eps \nu^2 n}{4k^2}$, so by Lemma~\ref{lem:chernoff} we have
\begin{align*}
  \mathbb{P}\left[ \Big| |\Gamma_Y \cap \Z| - \frac{\connectingconst \eps \nu^2 n}{4k^2}
  \Big| > \frac{\connectingconst \eps \nu^2 n}{8k^2} \right] \leq 2e^{-\connectingconst \eps \nu^2 n/96k^2}.
\end{align*}
Taking a union bound over all of the $\binom{n}{k-1}$ sets $Y$ of $k-1$ vertices of $H$, we find that with probability at least $3/4$ the collection $\Z$ contains at least $\frac{\connectingconst \eps \nu^2 n}{8k^2}$ absorbing sets for each $Y$.

We may therefore fix a collection $\Z$ such that each of the three events described as having probability at least $3/4$ hold. Let 
$\Z'$ be the subfamily of $\Z$ obtained by deleting the at most $\frac{\connectingconst \eps \nu^2 n}{100 k^2}$ members of $\Z$ which intersect another member of $\Z$, as well as any member of $\Z$ which is not an absorbing set for any $Y \in \binom{V(H)}{k-1}$, and define $t = |\Z'|$. Then $t \leq |\Z| \leq \frac{\connectingconst \eps \nu n}{k^2} \leq \connectingconst \eps n$ and, for any $Y \subseteq V(H)$ with
$|Y| = k-1$, at least $\frac{\connectingconst \eps \nu^2 n}{8 k^2} - \frac{\connectingconst \eps \nu^2 n}{100k^2} \geq \frac{\connectingconst \nu^2 \eps n}{10k^2} = \absorbingconst \eps n$ members of $\Z'$ are absorbing sets for $Y$.

Write $\Z' = \{Z_1, \dots, Z_t\}$ and $A = \bigcup \Z'$. For each $i$, since $Z_i$ is an absorbing set for some $Y \in \binom{V(H)}{k-1}$, we know that $Z_i$
induces a copy of $F \sm S$ in $H$; let $u_i$ and $v_i$ be the vertices corresponding to $u$ and $v$ in this copy. In particular $H[Z_i]$ contains a spanning path $P_i$ with endvertex pair $(u_i, v_i)$. We now apply Lemma~\ref{connecting} with $\sqrt{\eps}$ in place of $\eps$ to obtain vertex-disjoint paths $Q_i$ of length $3$
for $1 \leq i \leq t-1$ such that the endvertex pair of $Q_i$ is $(v_i, u_{i+1})$, and such that the paths 
$Q_i$ contain no vertices of $A$ except for the vertices of these endvertex pairs. 
This is possible because $|A| \leq t(3k^2-8k+6) \leq \sqrt{\eps} n$ and
$t \leq \connectingconst \eps n$.
Having chosen the paths $Q_i$, we define the path $P = P_1, Q_1, P_2, Q_2, \dots, P_{t-1}, Q_{t-1}, P_t$, so $P$ has $t (3k^2 - 8k + 6) + (t-1)(3k-4) \leq \eps n$ vertices and endvertex pair $(u, v)$, where $u=u_1$ and $v = v_t$.

Now consider any set $X \subseteq V(H) \sm V(P)$ such that $|X| \leq \absorbingconst \eps n$ and 
$|X|$ is divisible by $k-1$.
Partition $X$ arbitrarily into sets $Y_1, \dots, Y_{|X|/(k-1)}$ of size $k-1$ and greedily assign each $Y_i$ to a unique member $Z_{f(i)}$ of $\Z'$ which is an absorbing set
for $Y_i$. Such an assignment is possible since for each $Y_i$ at least $\absorbingconst \eps n$ members of $\Z'$ are absorbing sets for $Y$.
Since $Z_{f(i)}$ is an absorbing set for $Y_i$ there is then a path $P^*_{f(i)}$ in $H$ with $V(P^*_{f(i)}) = V(P_{f(i)}) \cup Y_i$ which has the same endvertex pair $(u_{f(i)}, v_{f(i)})$ as $P_{f(i)}$. 
By replacing the path $P_{f(i)}$ with $P^*_{f(i)}$ in $P$ for each $1 \leq i \leq t$ we obtain a path $P^*$ in $H$ with the same endvertex pair $(u, v)$ as $P$ such that $V(P^*) = V(P) \cup X$, as required.
\end{proof} 

\subsection{The path cover lemma} \label{sub:pathcover} 

Since a quasirandom $k$-graph remains quasirandom even after the deletion of almost all of its vertices,
the main difficulty in proving the path cover lemma is to show that $H$ contains a single path of
linear length. We can then greedily choose and delete paths to obtain the desired collection of
paths.

\begin{lemma} \label{lem:linearpath}
  For any $k \geq 2$, $0 < p < 1$ and any $0 < \eps < \frac{p}{2k \cdot k!}$, there exists $\mu > 0$ and
  $n_0$ such that the following holds.  Let $H$ be an $(n,p,\mu)$ $k$-graph with $n \geq n_0$, and 
  let $X \subseteq V(H)$ be such that $|X| \geq \eps^2 n$. 
  Then $H[X]$ contains a path of length at least $\eps^3 n$.
\end{lemma}

\begin{proof}
Let $\mu = \frac{p}{2} \eps^{2k}$, let $H' = H[X]$, and let $m = |X|$ (so $m \geq \eps^2n$). We first claim that the 
average vertex degree of vertices in $H'$ is at least $\frac{p}{2} \frac{m^k}{(k-1)!}$.  
Indeed, since $H$ is $(p,\mu)$-dense,
$e(X,\dots,X) \geq pm^k - \mu n^k \geq \frac{p}{2} m^k$.  Since $|H'| = \frac{1}{k!}
e(X,\dots,X)$, the average vertex degree of $H'$ is at least $\frac{p}{2} \frac{m^{k-1}}{(k-1)!}$. 
It follows that there exists a subgraph $H''$ of $H'$ with minimum vertex degree $\delta_1(H'') \geq \frac{p}{2k}
\frac{m^{k-1}}{(k-1)!}$, since the deletion of any vertex whose vertex degree is smaller than this increases the average vertex degree.

Now let $P$ be a longest path in $H''$ and let $(x,y)$ be an endvertex pair of $P$. Then all edges 
of $H''$ containing $y$ must also contain another vertex of $P$ or we could extend $P$. If the length of $P$ is less than $\eps^3 n$, then $P$ has fewer than $k \eps^3 n$ vertices, which implies that $d_{H''}(y) \leq k \eps^3 n m^{k-2} \leq
k\eps m^{k-1}$. Since $\eps < \frac{p}{2k^2 (k-1)!}$,
this contradicts the minimum degree of $H''$. So $P$ is a path in $H[X]$ of length at least $\eps^3 n$.
\end{proof}

\begin{proof}[Proof of Lemma~\ref{pathcover}]
We repeatedly apply Lemma~\ref{lem:linearpath} to choose a path of length at least $\eps^3 n$ in $H$ to add to $\mathcal{P}$. In each application we take $X$ to be the set of all vertices of $H$ not covered by previously-chosen members of $\mathcal{P}$, and we continue until $|X| < \eps^2 n$, at which point we can no longer apply Lemma~\ref{lem:linearpath}. At this point at most $\eps^2 n$ vertices of $H$ are not covered by paths in $\mathcal{P}$. Moreover, the paths in $\mathcal{P}$ are vertex-disjoint by our choice of $X$, and $\mathcal{P}$ therefore has size at most $1/\eps^3$ since each member of $\mathcal{P}$ has length at least $\eps^3 n$. 
\end{proof}

This completes the proof of the last of our three key lemmas, and so concludes the proof of Theorem~\ref{main}.

\section{Avoiding Hamilton $\ell$-cycles} \label{sec:constr}

In this section we prove Proposition~\ref{prop:constr} using the following construction, which was presented for $3$-graphs in~\cite{pp-lenz14}.

\begin{constr}
For integers $\ell$, $k$ and $n$ with $2 \leq \ell \leq k-1$, we form a random $k$-graph $H = H(n, k, \ell)$ on $n$ vertices as follows. Let $X$ and $Y$ be disjoint sets such that $|X \cup Y| = n$, $n/2 - 1 \leq |X| < n/2+1$, and $|X|$ is odd. Let $V := X \cup Y$, and let $G = G^{(\ell)}(n,\frac{1}{2})$ be the random $\ell$-graph on vertex set $V$ in which each edge is included with probability $1/2$, independently of all other choices. We take $V$ to be the vertex set of $H$, and say that a $k$-tuple $e$ of vertices of $V$ forms an edge of $H$ if either 
\begin{enumerate}[noitemsep, label=(\alph*)]
\item $|e \cap X|$ is even and $e$ induces a clique in $G$, or
\item $|e \cap X|$ is odd and $e$ induces an independent set in $G$.
\end{enumerate}
\end{constr}

The following proposition shows that, whenever $k-\ell$ divides $k$ and $2k$ divides $n$, the graph $H$ constructed above does not contain a Hamilton $\ell$-cycle.

\begin{prop} \label{lem:notwocycle}
Let $k \geq 3$ and $2 \leq \ell \leq k-1$ be integers such that $k-\ell$ divides $k$. Then for any integer $n$ which is divisible by $2k$ the $k$-graph $H = H(n, k, \ell)$ does not contain a Hamilton $\ell$-cycle.
\end{prop}

\begin{proof} 
Let $L := n/(k-\ell)$. Suppose that $C$ is a Hamilton $\ell$-cycle in $H$ with edges $e_1, \dots, e_L$ indexed in the order they appear in $C$. For each $1 \leq i \leq L$ define the \emph{block} $B_i := e_i \sm e_{i+1}$ (with indices taken modulo $L$), so $|B_i| = k-\ell$. Then for any $1 \leq i \leq L$ we have $e_i = B_i \cup \dots \cup B_{i+k/(k-\ell)-1}$ and $e_{i+1} = B_{i+1} \cup \dots \cup B_{i+k/(k-\ell)}$. Since $|e_i \cap e_{i+1}| = \ell$, by definition of $H$ either both $e_i$ and $e_{i+1}$ induce cliques in $G$, or both $e_i$ and $e_{i+1}$ induce independent sets in $G$. In either case we find that both $|e_i \cap X|$ and $|e_{i+1} \cap X|$ have the same parity, and therefore that $|B_i \cap X|$ and $|B_{i+k/(k-\ell)} \cap X|$ have the same parity. Since $L(k-\ell)/k = n/k$ is even, it follows that for any $1 \leq i \leq k/(k-\ell)$, the set
$$D_i = \bigcup_{0 \leq j < L(k-\ell)/k} B_{i + jk/(k-\ell)}$$
has the property that $|D_i \cap X|$ is even. However, since the sets $D_1, \dots, D_{k/(k-\ell)}$ partition the vertex set $V$, it follows that $|V \cap X|$ is even also, contradicting the fact that $|X|$ is odd.
\end{proof}

\begin{proof}[Proof of Proposition~\ref{prop:constr}]
By Proposition~\ref{lem:notwocycle} it suffices to show that for any $\mu > 0$ there exists $n_0$ such that if $n \geq n_0$ then the random graph $H(n, k, \ell)$ satisfies properties (a) and (b) of Proposition~\ref{prop:constr} with positive probability. For the case $k=3$ short proofs of these statements were given in \cite[Lemmas 20 and 21]{pp-lenz14}, and similar arguments hold for any $k \geq 3$ (we omit the details).
\end{proof} 

We conclude by giving the construction mentioned in Section~\ref{sec:new}, of an $(n, p, \mu)$ $k$-graph~$H'$ with $\delta_{j}(H') \geq \alpha \binom{n}{k-j}$ for any $1 \leq j \leq \ell-1$ which contains a vertex $x$ such that the intersection of any edge containing $x$ and any edge not containing $x$ has size at most $\ell-1$ (where $2 \leq \ell \leq k-1$); this is similar to the construction given above. Fix a set $V$ of $n$ vertices, and choose some $x \in V$. Let $G = G^{(\ell)}(n,\frac{1}{2})$ be the random $\ell$-graph on vertex set $V$ in which each edge is included with probability $1/2$, independently of all other choices. We define $H'$ to be the $k$-graph on $V$ whose edges are all $k$-tuples $e$ of vertices of $V$ such that either $x \in e$ and $e$ induces a clique in $G$, or $x \notin e$ and $e$ induces an independent set in $G$. Since a clique and an independent set in an $\ell$-graph can have at most $\ell-1$ vertices in common, $H'$ has the desired property that the intersection of any edge containing $x$ and any edge not containing $x$ has size at most $\ell-1$. Moreover, standard probabilistic arguments similar to those of \cite[Lemmas 20 and 21]{pp-lenz14} show that, for any fixed $\mu > 0$, with high probability $H'$ is indeed an $(n, 2^{-\binom{k}{\ell}}, \mu)$ $k$-graph with $\delta_{j}(H') \geq 2^{-\binom{k}{\ell}} \binom{n}{k-j}$ for any $1 \leq j \leq \ell-1$.

\medskip \noindent
\textbf{Acknowledgements.} We would like to thank Peter Allen for helpful discussions, in particular for pointing out the modified construction given in the final paragraph. We also thank the anonymous referees for their helpful comments.

\bibliographystyle{abbrv}
\bibliography{refs}

\end{document}